\documentclass[11pt,reqno]{amsart}

\usepackage{mathrsfs}
\usepackage{amsfonts}
\usepackage{amssymb}
\usepackage{amsthm}
\usepackage{amsmath}
\usepackage{latexsym}
\usepackage{tikz}
\usepackage{geometry}

\newtheorem{theorem}{Theorem}[section]
\newtheorem{proposition}[theorem]{Proposition}
\newtheorem{lemma}[theorem]{Lemma}

\newtheorem{remark}[theorem]{Remark}

\makeatletter \@addtoreset{equation}{section} \makeatother

\begin{document}

\title[Global Classical Solutions of  3D MHD System on Periodic Boxes]
{ Global Classical Solutions of 3D Viscous MHD System without magnetic diffusion on Periodic Boxes}

\author[R. Pan]{Ronghua Pan}\address{School of Mathematics,
Georgia Institute of Technology, Atlanta, GA 30332 USA}
\email{\tt panrh@math.gatech.edu}
\author[Y. Zhou]{Yi Zhou}\address{School of Mathematical Sciences, Fudan University, Shanghai 200433, P.R.China}
\email{\tt yizhou@fudan.edu.cn}
\author[Y. Zhu]{Yi Zhu*} \address{Department of Mathematics, East China University of Science and Technology, Shanghai 200237, P.R.China}
\email{\tt zhuyim@ecust.edu.cn}

\footnotetext[1]{Corresponding author.}

\date{}

\begin{abstract}
In this paper, we study the global existence of classical solutions to the three dimensional incompressible viscous magneto-hydrodynamical system without magnetic diffusion on periodic boxes, i.e., with periodic boundary conditions. We work in Eulerian coordinate and employ a time-weighted energy estimate to prove the global
 existence result, under the assumptions that the initial magnetic field is close enough to an equilibrium state and the initial data have some symmetries.
\end{abstract}


\maketitle

\section{Introduction}
The equations of viscous magnetohydrodynamics (MHD) model the motion of electrically conducting fluids interacting with magnetic fields. When the fluids are strongly collisional plasmas, or the resistivity due to collisions is extremely small, the diffusion in magnetic field is often neglected \cite{HC, TGCDP, LDEM}. When magnetic diffusion is missing, it is extremely interesting to understand whether the fluid viscosity only could prevent singularity development from small smooth initial data in three dimensional physical space, in view of the strongly nonlinear coupling between fluids and magnetic field. Mathematically, it is also close in structure to the model of dynamics of certain complex fluids,  including hydrodynamics of viscoelastic fluids, c.f. \cite{FL,FLCLPZ,FLLX,FLPZ,FLTZ}. To this purpose, we investigate the global existence of smooth solutions to the following initial boundary value problem
\begin{equation}\label{eq1.1}
\begin{cases}
B_t + u \cdot \nabla B = B\cdot \nabla u,\\
u_t + u \cdot \nabla u - \Delta u + \nabla p= B\cdot \nabla B,\\
\nabla \cdot u = \nabla \cdot B = 0,\\
u(0,x)= u_0(x), \quad B(0,x) = B_0(x),
\end{cases}
\end{equation}
with periodic boundary conditions
\begin{equation}\label{eq0.2}
x\in [-\pi, \pi]^3 = \mathbb{T}^3.
\end{equation}
here $B=(B_1,B_2,B_3)$ denotes the magnetic field, $u=(u_1,u_2,u_3)$ the fluid velocity, $p=q+\frac{1}{2}|B|^2$ where $q$ denotes the scalar pressure of the fluid.

Extensively impressive progresses had been achieved in the past decades for MHD systems. Indeed, according to the level of dissipations, there are roughly three different layers of models: inviscid and non-resistive (no viscosity, no magnetic diffusion, hence no dissipation); viscous and resistive (fully dissipative in fluids and in magnetic field); and partially dissipative (only viscosity or magnetic diffusion presents). On one hand, it is natural to expect global existence of classical solutions for viscous and resistive MHD at least for small initial data, this has been confirmed in classical papers by Duvaut and Lions \cite{DL} and by Sermange and Temam \cite{ST}. In 2008, Abidi and Paicu \cite{AP} generalized these results to the inhomogeneous MHD system with initial data in the so-called critical spaces. More recently, Cao and Wu \cite{CW}, also see \cite{CRW}, proved the global well-posedness for any data in $H^2(\mathbb{R}^2)$ with mixed partial viscosity and magnetic diffusion in two dimensional MHD system. On the other hand, it is somehow striking that Bardos, Sulem and Sulem \cite{CBCS} proved that  the inviscid and non-resistive MHD system also admits a unique global classical solution when the initial data is near a nontrivial equilibrium. It seems that purely dispersion and some coupling of nonlinearity between fluids and magnetic field are sufficient to maintain the regularity from initial data. Very recently, the vanishing dissipation limit from
fully dissipative MHD system to inviscid and non-resistive MHD system has been justified by \cite{HXY, CL, WZ} under some structural conditions between viscosity and magnetic diffusion coefficients. Therefore, it is not surprise why the remaining case, partially dissipative MHD, attracts a lot of attentions in the recent years. As documented in \cite{CW, LZ}, the inviscid and resistive 2D MHD system admits a global $H^1$ weak solution, but the uniqueness of such solution with higher order regularity is still not known yet.

In the case of our consideration, namely the incompressible MHD system with positive viscosity and zero resistivity, it is still an open problem whether or not there exists global classical solutions even in two dimensional space for generic smooth initial data. The main difficulty of studying these MHD systems lies in the non-resistivity of the magnetic equation. Some interesting results have been obtained for small smooth
solutions. For a closely related model in 3D, the global well-posedness was established by Lin and Zhang \cite{FLPZ}, and a simpler proof was offered by Lin and Zhang \cite{FLTZ}. With certain admissible condition for initial data, Lin, Xu and Zhang \cite{FLLX} established the global existence in 2D for initial data close to an nontrivial equilibrium state and the 3D case was proved by Xu and Zhang \cite{LXPZ}. Later, Ren, Wu, Xiang and Zhang \cite{REN} removed the restriction in 2D case (see Zhang \cite{TZ} for a simplified proof). We also refer to another proof for 2D incompressible case by Hu and Lin \cite{HU}. Hu \cite{HU2} further established some results for 2D compressible MHD system.
Very recently, under Lagrangian coordinate system, Abidi and Zhang \cite{HAPZ} proved the global well-posedness for 3D MHD system without the admissible restriction. For somehow large data, Lei \cite{Lei} proved the global regularity of some axially symmetric solutions in three space dimensional MHD system. While all results down the line are about Cauchy problem, an initial boundary value problem for 2D case under Eulerian coordinate in a strip domain $\mathbb{R}\times(0,1)$ was done by Ren, Xiang and Zhang \cite{RXZ2} recently. And the 3D case on $\mathbb{R}^2\times(0,1)$ for both compressible and incompressible fluids was considered by Tan and Wang \cite{TW} under Lagrangian coordinate. In the three dimensional case, these inspiring results, along with many innovative methods and estimates, made full use of partial dissipation offered by viscosity, dispersion of waves on unbounded domain and the structure of Lagrangian formulation (which contains one time derivative already and helps capture the weak dissipation). It is then natural to explore the following two questions. Is it possible to establish global existence of small smooth solutions on bounded domain, where the dispersion effect is limited? Can one work with Eulerian coordinate where the system takes a simpler form with the cost of the loss of one time derivative, thus the loss of possible time decay?

Our main aim in this paper is to offer answers to these questions. Indeed, as one step in this direction, we will establish the global existence of small smooth solutions to the 3D incompressible viscous magneto-hydrodynamical system without resistivity on periodic boxes, under the assumptions that the initial magnetic field is close enough to an equilibrium state and the initial data have some symmetry structure. We will also avoid the use of Lagrangian formulation. The advantage of the Eulerian coordinate is that, if successful, it would be neat and simple.

To fix the idea, we adopt the following notations
\begin{equation}\nonumber
x_h=(x_1,x_2), \quad \nabla_h=(\partial_1, \partial_2), \quad B_h=(B_1, B_2)^{\top},
\end{equation}
and similar notations for other quantities without causing further confusions.

We assume that
\begin{equation}\label{eq0.3}
\begin{split}
u_{0,h}(x), \quad B_{0,3}(x)\quad &\text{are even periodic with respect to } x_3,\\
u_{0,3}(x),\quad  B_{0,h}(x) \quad &\text{are odd periodic with respect to } x_3,
\end{split}
\end{equation}
moreover
\begin{equation}\label{eq0.4}
\int_{\mathbb{T}^3} u_0 \;dx = 0, \quad \int_{\mathbb{T}^3} B_{0,3} \;dx = \alpha \neq 0.
\end{equation}

Our main result can be stated as follows.
\begin{theorem}\label{thm1}
Consider the 3D MHD system \eqref{eq1.1}-\eqref{eq0.2} with initial data satisfies the conditions \eqref{eq0.3}-\eqref{eq0.4}. Then, there exists a small constant $\varepsilon > 0$, only depending  on $\alpha$ such that the system \eqref{eq1.1} admits a
global smooth solution provided that
\begin{equation}\nonumber
\|u_0\|_{H^{2s+1}} + \|\nabla B_0\|_{H^{2s}}\leq \varepsilon,
\end{equation}
where $s \geq 5$ is an integer.
\end{theorem}

\begin{remark}
  Our methods can be used to other related models. Similar result for the compressible system will be presented in a forthcoming paper.
\end{remark}

Without loss of generality, we assume $\alpha = (2\pi)^3$, and following Lin and Zhang \cite{FLPZ}, we let
$$ B_0=b_0 + e_3,$$
where $e_3=(0,0,1)^{\top}$.
Hence, we have
\begin{equation}\label{eq1.02}
\int_{\mathbb{T}^3} b_0 \;dx = \int_{\mathbb{T}^3} u_0 \;dx = 0.
\end{equation}

Set $B=b+e_3$, we get the system of pair $(u, b)$ as follows
\begin{equation}\label{eq1.2}
\begin{cases}
b_t + u \cdot \nabla b = b\cdot \nabla u + \partial_3 u,\\
u_t + u \cdot \nabla u - \Delta u + \nabla p= b\cdot \nabla b + \partial_3 b,\\
\nabla \cdot u = \nabla \cdot b = 0,
\end{cases}
\end{equation}
with initial data
$$u(0,x)= u_0(x), \quad b(0,x) = b_0(x).$$
and the property of initial data \eqref{eq0.3} will be hold, i.e.,
\begin{equation}\label{eq1.3}
\begin{split}
u_h(0,x), \quad b_3(0,x)\quad &\text{are even periodic with respect to } x_3,\\
u_3(0,x),\quad  b_h(0,x) \quad &\text{are odd periodic with respect to } x_3.
\end{split}
\end{equation}
Also, by the periodic boundary conditions \eqref{eq1.02} and system \eqref{eq1.2}, we have
\begin{equation}\label{eq1.4}
\int_{\mathbb{T}^3} b \;dx = \int_{\mathbb{T}^3} u \;dx = 0.
\end{equation}

\begin{remark}
  The property \eqref{eq1.3} will hold in the time evolution. Indeed, we can define $\bar u(t,x), \bar b(t,x) $ as follows
  \begin{equation}\nonumber
\begin{split}
\bar u_h(t, x_h, x_3) =& u_h(t,x_h, - x_3), \quad \bar u_3(t,x_h, x_3) = - u(t,x_h,- x_3), \\
\bar b_h(t, x_h, x_3) =& - b_h(t,x_h, -x_3), \quad \bar b_3(t,x_h, x_3) = b_3 (t, x_h, -x_3).
\end{split}
\end{equation}
  Then, quantities $\bar u, \bar b$ satisfy the same system \eqref{eq1.2} like $u, b$ and also have the same initial data. Hence, by the uniqueness of classical solution, we obtain $\bar b (t,x) = b(t,x)$ and $\bar u(t,x) = u(t,x)$. Therefore, we see the property \eqref{eq1.3} persist.
\end{remark}

Although the property \eqref{eq1.3} (from \eqref{eq0.3} ) could be realized physically initially and is preserved in time evolution as explained in the previous remark, its physical interpretation is not quite clear. It identifies a significant class of initial data on periodic box admitting global classical solution to \eqref{eq1.2} near a nontrivial magnetic equilibrium. Mathematically, it helps in analysis  allow us to use Poincar\'{e} inequality
(Proposition \ref{prop}) for some crucial terms in the estimates. On the other hand, it is also needed to rule out an extremely unclear situation related to the global regularity of 2D MHD without magnetic diffusion when the initial data is a small perturbation near the {\bf trivial} equilibrium, which is a very difficult problem.  To help the readers understand the situation, let's choose the following special class of initial data
\begin{equation}\label{specialdata} B_0(x)=(B_0^h(x_h), 1), u_0(x)=(u_0^h(x_h),0), \nabla_h \cdot B_0^h = \nabla_h \cdot u_0^h = 0, \end{equation}
which reduces the original system \eqref{eq1.1} into the following 2D problem
\begin{equation}\label{MHD2D}
\begin{cases}
B^h_t + u^h \cdot \nabla_h B^h = B^h \cdot \nabla_h u^h,\\
u^h_t + u^h \cdot \nabla_h u^h - \Delta_h u^h + \nabla_h P = B^h \cdot \nabla_h B^h,\\
\nabla_h \cdot B^h = \nabla_h \cdot u^h = 0,\\
B^h(0, x_h)=B_0^h(x_h), u^h(0,x_h)=u_0^h(x_h).
\end{cases}
\end{equation}
with initial data $(B_0^h(x_h), u_0^h(x_h)) $ being a perturbation near a trivial equilibrium. This, however, is still a challenging open problem. Furthermore, we note that if
$(B^h, u^h)(t, x_h)$ is a classical solution of \eqref{MHD2D}, then $B(t, x)=(B^h(t, x_h), 1), u(t, x)=(u^h(t, x_h), 0)$ is the corresponding solution of
 \eqref{eq1.1} with initial data \eqref{specialdata}.  When the problem is considered in the whole space, by requiring finite energy in $\mathbb{R}^3$, one finds
$B_0^h(x_h)=0= u_0^h(x_h) $, avoiding the complex situation successfully. However, on a periodic bounded domain, finite energy condition is not sufficient to show the solution of \eqref{MHD2D} is trivial. To avoid the unclear situation mentioned before, some additional conditions are needed. In this paper, we impose \eqref{eq1.3} (from \eqref{eq0.3}) to
ensure that $B_0^h(x_h)=0= u_0^h(x_h) $, and thus system \eqref{MHD2D} has only trivial solution. It would be interesting to explore other conditions for this purpose.

  In order to prove Theorem \ref{thm1}, we only need to consider the system \eqref{eq1.2} instead.

  In this paper, we have to face the difficulties from bounded domain and the loss of weak dissipation without using Lagrangian formulation.
  One of the major differences in analysis between the whole space and the bounded domain is the character of dissipation. For the whole space, although the system contains only the viscosity, it is possible to recover dissipative structure for all components of $u$ and $b$, in addition to the advantage of wave dispersion. For the bounded domain, however, it is extremely difficult to recover dissipative structure for all components of
  $u$ and $b$. Indeed, even with the help of  condition \eqref{eq1.3} and  Poincar\'{e} inequality
(Proposition \ref{prop}), we could not derive dissipation for $b_3$. We emphasize that the analysis of the whole space case is quite complicated and exhibits very different features comparing to our case here. They are different in nature, also difficult in different aspects.
These challenges will be overcome through a carefully designed weighted energy method with the help of some observations to the structure of the system. One of the major observations is that the time derivative of
$b$ is essentially quadratic terms plus a derivative term in the good direction $x_3$ where dissipation kicks in. Another observation is that although the bounded domain pushes us away from possible dispersion of
waves, it does compensate us with Poincar\'{e} inequality. However, the high space dimensions, the lack of
magnetic diffusion, and the strongly coupled nonlinearity of the problem, make the mathematical analysis very challenging. Even with our carefully designed time-weighted energies, there are still many dedicated technical issues. One of  our main obstacles is to derive the time dissipative estimate to the term $b\cdot\nabla b$ which behaves most wildly in the system. Writing $b\cdot\nabla b = b_3 \cdot \partial_3 b + b_h \cdot \partial_h b$, we notice that $b_3\cdot \partial_3 b$ contains one good quantity $\partial_3 b$ can be estimated relatively easily due to dissipation in $x_3$ direction. Hence, we focus on the term $b_h \cdot \nabla_h b$ containing two bad terms. To overcome this difficulty, we make full use of the condition \eqref{eq1.3} and Poincar\'{e} inequality in $x_3$ direction. Thus the norm of $b_h$ can be controlled by the norm of $\partial_3b_h$.  This specific choice of estimate avoids the presence of interaction between two wild quantities. Such an idea actually origins from \emph{null condition} in the theory of wave equations. However, we still have to come across other difficulties in the estimate process. For example, we can not achieve the uniform bound of all higher order norms that we wanted. Instead, we turn to control the growth of such norms by the energy frame we construct in the next section. More detailed decay estimates will also be presented in Section 2.

\section{Energy estimate and the proof of  main result}

\subsection{Preliminary}
In this subsection, we first introduce a useful proposition related to Poincar\'{e} inequality  which plays an important role in our proof to the main theorem of this paper.
\begin{proposition}\label{prop}
For any function $f(x) \in H^{k+1}(\mathbb{T}^3)$, $k \in \mathbb{N}$ satisfying the following condition
\begin{equation}\label{eqp1}
  \frac{1}{2\pi}\int_{-\pi}^{\pi} f(x_h,x_3)\ dx_3 =0, \qquad \forall \; x_h \in \mathbb{T}^2,
\end{equation}
it holds that
\begin{equation}\nonumber
 \|f\|_{H^k (\mathbb{T}^3)} \lesssim  \|\partial_3f\|_{H^k (\mathbb{T}^3)} .
\end{equation}
\end{proposition}

\begin{proof}

First, we can write
\begin{equation}\label{eqp2}
\|f\|^2_{H^k (\mathbb{T}^3)} =\sum_{|\alpha|=0}^k \int_{\mathbb{T}^2}\int_{-\pi}^{\pi} |\partial^{\alpha}f(x_h,x_3)|^2 dx_3 dx_h  .
\end{equation}
Here, $\alpha = (\alpha_1, \alpha_2, \alpha_3)$ is a multi-index and $\partial ^{\alpha} = \partial_1^{\alpha_1}\partial_2^{\alpha_2}\partial_3^{\alpha_3}$.

Notice the condition \eqref{eqp1}, we have, for multi-index $\alpha=(\alpha_1,\alpha_2,0)$:
\begin{equation}\nonumber
  \frac{1}{2\pi}\int_{-\pi}^{\pi} \partial^{\alpha}f(x_h,x_3) dx_3 =0,
\end{equation}
And for multi-index $\alpha=(\alpha_1,\alpha_2,\alpha_3)$ where $\alpha_3 >0$, by the periodic boundary condition, we also have
\begin{equation}\nonumber
  \int_{-\pi}^{\pi} \partial^{\alpha}f(x_h,x_3) dx_3 = \partial_1^{\alpha_1}\partial_2^{\alpha_2}\partial_3^{\alpha_3-1} f(x_h,\cdot)\Big|_{-\pi}^{\pi} =0.
\end{equation}

Therefore, the average value of $\partial^{\alpha} f(x_h, \cdot)$ in $x_3$ direction over $[-\pi, \pi]$ is zero.
Hence, applying standard Poincar\'{e} inequality to $\partial^{\alpha}f(x_h,x_3)$ in the $x_3$ direction, we have $\forall x_h \in \mathbb{T}^2$:
\begin{equation}\nonumber
  \int_{-\pi}^{\pi} |\partial^{\alpha}f(x_h,x_3)|^2 dx_3 \lesssim \int_{-\pi}^{\pi} |\partial^{\alpha}\partial_3f(x_h,x_3)|^2 dx_3.
\end{equation}
According to the definition of $\|f\|_{H^k(\mathbb{T}^3)}$ i.e. \eqref{eqp2}, we finally obtain
\begin{equation}\nonumber
\begin{split}
\|f\|^2_{H^k (\mathbb{T}^3)} =&\sum_{|\alpha|=0}^k \int_{\mathbb{T}^2}\int_{-\pi}^{\pi} |\partial^{\alpha}f(x_h,x_3)|^2 dx_3 dx_h \\
\lesssim & \sum_{|\alpha|=0}^k \int_{\mathbb{T}^2}\int_{-\pi}^{\pi} |\partial^{\alpha}\partial_3f(x_h,x_3)|^2 dx_3 dx_h \\
= &\|\partial_3f\|^2_{H^k (\mathbb{T}^3)}.
\end{split}
\end{equation}

\end{proof}

Now, let us introduce the energy frame that will enable us to achieve our desired estimate.  Based on our discussion in Section 1, we define some time-weighted energies for the system \eqref{eq1.2}. The energies below are defined on the domain $\mathbb{R}^+ \times\mathbb{T}^3$. For $s\in \mathbb{N}$ and $0<\sigma<1$, we set
\allowdisplaybreaks[2]
\begin{align}\label{eqframe}
E_0(t) =& \sup_{0\leq \tau \leq t} (1+\tau)^{-\sigma} (\|u(\tau)\|_{H^{2s+1}}^2+\|b(\tau)\|_{H^{2s+1}}^2)\nonumber\\
&+ \int_{0}^{t}  (1+\tau)^{-1-\sigma} (\|u(\tau)\|_{H^{2s+1}}^2+\|b(\tau)\|_{H^{2s+1}}^2) \; d\tau\nonumber\\
&+ \int_{0}^{t}  (1+\tau)^{-\sigma}(\|u(\tau)\|_{H^{2s+2}}^2+\|\partial_3 b(\tau)\|_{H^{2s}}^2) \; d\tau,\nonumber\\
G_0(t) =& \sup_{0\leq \tau \leq t} (1+\tau)^{1-\sigma} (\|\partial_3 u(\tau)\|_{H^{2s}}^2+\|\partial_3 b(\tau)\|_{H^{2s}}^2)\nonumber\\
&+ \int_{0}^{t}  (1+\tau)^{1-\sigma}\|\partial_3 u(\tau)\|_{H^{2s+1}}^2 \; d\tau,\nonumber\\
G_1(t) =& \sup_{0\leq \tau \leq t} (1+\tau)^{3-\sigma} (\|\partial_3 u(\tau)\|_{H^{2s-2}}^2+\|\partial_3 b(\tau)\|_{H^{2s-2}}^2)\nonumber\\
&+ \int_{0}^{t}  (1+\tau)^{3-\sigma}\|\partial_3 u(\tau)\|_{H^{2s-1}}^2 \; d\tau,\nonumber\\
E_1(t) =& \sup_{0\leq \tau \leq t} (1+\tau)^{3-\sigma} \|u(\tau)\|_{H^{2s-2}}^2\nonumber\\
&+ \int_{0}^{t}  (1+\tau)^{3-\sigma}(\|u(\tau)\|_{H^{2s-1}}^2+\|\partial_3 b(\tau)\|_{H^{2s-3}}^2) \; d\tau,\nonumber\\
e_0(t) =& \sup_{0\leq \tau \leq t} \|b(\tau)\|_{H^{2s}}^2.
\end{align}

In the following, we will successively derive the estimate of each energy stated above. By \eqref{eq1.4} and Poincar$\mathrm{\acute{e}}$ inequality, we only need to consider the highest order norms in each energy.

\subsection{\textit{A priori} estimate}
First, we will deal with the highest order energy, i.e., $E_0(t)$. It shows that the highest order norm $H^{2s+1}(\mathbb{T}^3)$ of $u(t,\cdot)$ and $b(t,\cdot)$ will grow in the time evolution.
\begin{lemma}\label{lem1}
Assume that $s\geq 5$ and the energies are defined as in \eqref{eqframe}, then we have
\begin{equation}\label{E_0}\nonumber
E_0(t)\lesssim E_0(0)+E_0(t)E_1^{1/2}(t)+ E_1^{1/2}(t)e_0(t)+E_0^{5/6}(t)E_1^{1/6}(t)e_0^{1/2}(t)+ E_0(t)e_0^{1/2}(t).
\end{equation}
\end{lemma}

\begin{proof}
We divide the proof into two steps. Instead of deriving the estimate of $E_0$ directly, we shall first get the estimate of $E_{0,1}(t)$ defined by

\begin{equation}\label{E01}
\begin{split}
E_{0,1}(t)\triangleq&\sup_{0\leq \tau \leq t} (1+\tau)^{-\sigma} (\|u(\tau)\|_{H^{2s+1}}^2+\|b(\tau)\|_{H^{2s+1}}^2)
+ \int_{0}^{t}  (1+\tau)^{-\sigma}\|u(\tau)\|_{H^{2s+2}}^2 \; d\tau\\
&\qquad+ \int_{0}^{t}  (1+\tau)^{-1-\sigma} (\|u(\tau)\|_{H^{2s+1}}^2+\|b(\tau)\|_{H^{2s+1}}^2) \; d\tau .
\end{split}
\end{equation}
$\mathbf{Step\;\;1}$

Applying $\nabla^{2s+1}$ derivative on the system \eqref{eq1.2}. Then, taking inner product with $\nabla^{2s+1} b$ for the first equation of system \eqref{eq1.2} and taking inner product with $\nabla^{2s+1} u$ for the second equation of system \eqref{eq1.2}. Adding them up and multiplying the time weight $(1+t)^{-\sigma}$, we get
\begin{equation}\label{eqlem1.1}
\begin{split}
\frac{1}{2}\frac{d}{dt}&(1+t)^{-\sigma}(\|u\|_{\dot H^{2s+1}}^2+\|b\|_{\dot H^{2s+1}}^2)+ \frac{\sigma}{2} (1+t)^{-1-\sigma}(\|u\|_{\dot H^{2s+1}}^2+\|b\|_{\dot H^{2s+1}}^2)\\
&+(1+t)^{-\sigma}\|u\|_{\dot H^{2s+2}}^2
= I_1 + I_2 + I_3 + I_4,
\end{split}
\end{equation}
where,
\allowdisplaybreaks[2]
\begin{align}
I_1=&-(1+t)^{-\sigma}\int_{\mathbb{T}^3} u\cdot \nabla \nabla^{2s+1} u \; \nabla^{2s+1} u + u\cdot \nabla \nabla^{2s+1} b \; \nabla^{2s+1} b \;dx\nonumber\\
&+(1+t)^{-\sigma}\int_{\mathbb{T}^3} b\cdot \nabla \nabla^{2s+1} b \; \nabla^{2s+1} u + b\cdot \nabla \nabla^{2s+1} u \; \nabla^{2s+1} b \;dx\nonumber\\
&+(1+t)^{-\sigma}\int_{\mathbb{T}^3} \nabla^{2s+1}\partial_3 u \; \nabla^{2s+1} b + \nabla^{2s+1}\partial_3 b \;
\nabla^{2s+1}u \;dx,\nonumber\\
I_2=&-(1+t)^{-\sigma}\sum_{k=1}^{2s+1}\int_{\mathbb{T}^3} \nabla^{k}u\cdot \nabla \nabla^{2s+1-k} u\; \nabla^{2s+1}u \;dx,\nonumber\\
I_3=&(1+t)^{-\sigma}\sum_{k=1}^{s}\int_{\mathbb{T}^3}(\nabla^{k}b \cdot \nabla \nabla^{2s+1-k} u- \nabla^{k} u \cdot \nabla
\nabla^{2s+1-k} b) \nabla^{2s+1} b\;dx\nonumber\\
&+(1+t)^{-\sigma}\sum_{k=s+1}^{2s+1}\int_{\mathbb{T}^3}(\nabla^{k}b \cdot \nabla \nabla^{2s+1-k} u- \nabla^{k} u \cdot \nabla
\nabla^{2s+1-k} b) \nabla^{2s+1} b\;dx,\nonumber\\
I_4=&-(1+t)^{-\sigma}\sum_{k=1}^{s}\int_{\mathbb{T}^3}\nabla^{k}b \cdot \nabla \nabla^{2s+1-k} b\; \nabla^{2s+1} u\;dx\nonumber\\
&-(1+t)^{-\sigma}\sum_{k=s+1}^{2s+1}\int_{\mathbb{T}^3}\nabla^{k}b \cdot \nabla \nabla^{2s+1-k} b\; \nabla^{2s+1} u\;dx.\nonumber
\end{align}

We shall estimate each term on the right hand side of \eqref{eqlem1.1}. First, for the term $I_1$, using integration by parts and the divergence free condition, we have
\begin{equation}\label{eqI1}
I_1 = 0.
\end{equation}

The main idea for the next estimates is that we will carefully derive the bound of each term so that it can be controlled by the combination of energies defined in \eqref{eqframe}. By H\"{o}lder inequality and Sobolev imbedding theorem, we have
\begin{equation}\nonumber
\begin{split}
|I_2| \lesssim &(1+t)^{-\sigma}\|u\|_{W^{s+1,\infty}}\|u\|_{H^{2s+1}}^2\\
\lesssim & (1+t)^{-\sigma}\|u\|_{H^{s+3}}\|u\|_{H^{2s+1}}^2\\
\lesssim & (1+t)^{-\sigma}\|u\|_{H^{2s-1}}\|u\|_{H^{2s+1}}^2.
\end{split}
\end{equation}
provided that $s \geq 4$.
Hence,
\begin{equation}\label{eqI2}
\int_{0}^{t} |I_2(\tau)| \;d\tau \lesssim \sup_{0\leq \tau \leq t}  (1+\tau)^{-\sigma}\|u\|_{H^{2s+1}}^2 \int_{0}^{t} \|u\|_{H^{2s-1}}\;d\tau \lesssim E_0(t) E_1^{1/2}(t).
\end{equation}

Similarly, for the first part of $I_3$ (we denote the first term on the right hand as $I_{3,1}$ and the second term as $I_{3,2}$),  we see that
\begin{equation}\nonumber
\begin{split}
|I_{3,1}| \lesssim & (1+t)^{-\sigma}(\|b\|_{W^{s,\infty}} \|u\|_{H^{2s+1}} \|b\|_{H^{2s+1}} + \|u\|_{W^{s,\infty}}\|b\|_{H^{2s+1}}^2)\\
\lesssim & (1+t)^{-\sigma}(\|b\|_{H^{s+2}} \|u\|_{H^{2s+1}} \|b\|_{H^{2s+1}} + \|u\|_{H^{s+2}}\|b\|_{H^{2s+1}}^2)\\
\lesssim & (1+t)^{-\sigma}(\|b\|_{H^{2s}} \|u\|_{H^{2s+1}} \|b\|_{H^{2s+1}} + \|u\|_{H^{2s-1}}\|b\|_{H^{2s+1}}^2).
\end{split}
\end{equation}
provided that $s\geq 3$. And for the second part of $I_3$, we have
\begin{equation}\nonumber
\begin{split}
|I_{3,2}|\lesssim & (1+t)^{-\sigma} (\|b\|_{H^{2s+1}}^2 \|u\|_{W^{s+1,\infty}}+ \|u\|_{H^{2s+1}}\|b\|_{W^{s+1,\infty}}\|b\|_{H^{2s+1}})\\
\lesssim & (1+t)^{-\sigma} (\|b\|_{H^{2s+1}}^2 \|u\|_{H^{s+3}}+ \|u\|_{H^{2s+1}}\|b\|_{H^{s+3}}\|b\|_{H^{2s+1}})\\
\lesssim & (1+t)^{-\sigma} (\|b\|_{H^{2s+1}}^2 \|u\|_{H^{2s-1}}+ \|u\|_{H^{2s+1}}\|b\|_{H^{2s}}\|b\|_{H^{2s+1}}).
\end{split}
\end{equation}
provided that $s\geq 4$.
Hence, combining $I_{3,1}$ and $I_{3,2}$ together and using H\"{o}lder inequality, we get
\begin{equation}\label{eqI3.1}
\begin{split}
&\int_{0}^{t} |I_{3}(\tau)|\; d\tau \\
\lesssim &\sup_{0\leq \tau \leq t} \|b\|_{H^{2s}}\Big(\int_{0}^{t} (1+\tau)^{-1-\sigma}\|b\|_{H^{2s+1}}^2 \; d\tau \Big)^\frac{1}{2}
\Big(\int_{0}^{t} (1+\tau)^{1-\sigma}\|u\|_{H^{2s+1}}^2 \; d\tau \Big)^\frac{1}{2}\\
&+\sup_{0\leq \tau \leq t}(1+\tau)^{-\sigma} \|b\|_{H^{2s+1}}^2\int_{0}^{t}\|u\|_{H^{2s-1}} d\tau.
\end{split}
\end{equation}
Using Gagliardo--Nirenberg interpolation inequality and H\"{o}lder inequality, we can bound
\begin{equation}\label{eqI3.2}
\begin{split}
\int_{0}^{t} (1+\tau)^{1-\sigma}\|u\|_{H^{2s+1}}^2 \; d\tau \lesssim &
\int_{0}^{t} \big[(1+\tau)^{3-\sigma}\|u\|_{H^{2s-1}}^2\big]^\frac{1}{3} \big[(1+\tau)^{-\sigma}\|u\|_{H^{2s+2}}^2\big]^\frac{2}{3}\; d\tau \\
\lesssim &E_0^{2/3}(t)E_1^{1/3}(t).
\end{split}
\end{equation}
Thus, combining \eqref{eqI3.1} with \eqref{eqI3.2}, we finally obtain the estimate of $I_3$
\begin{equation}\label{eqI3}
\int_{0}^{t} |I_{3}(\tau)|\; d\tau \lesssim E_0^{5/6}(t)E_1^{1/6}(t)e_0^{1/2}(t) + E_1^{1/2}(t)e_0(t).
\end{equation}

Next, for the last term $I_4$, we use the same method as above and obtain
\begin{equation}\nonumber
\begin{split}
|I_4| \lesssim& (1+t)^{-\sigma}\|b\|_{H^{2s+1}}\|b\|_{W^{s+1,\infty}}\|u\|_{H^{2s+1}} \\
\lesssim& (1+t)^{-\sigma}\|b\|_{H^{2s+1}}\|b\|_{H^{2s}}\|u\|_{H^{2s+1}},
\end{split}
\end{equation}
provided that $s\geq 3$.
Hence,
\begin{equation}\label{eqI4}
\int_{0}^{t} |I_{4}(\tau)|\; d\tau \lesssim E_0^{5/6}(t)E_1^{1/6}(t)e_0^{1/2}(t).
\end{equation}

Summing up the estimates for $I_1\thicksim I_4$, i.e., \eqref{eqI1}, \eqref{eqI2}, \eqref{eqI3} and \eqref{eqI4}, and integrating \eqref{eqlem1.1} in time, we can get the estimate of $E_{0,1}(t)$ which is defined in \eqref{E01}
\begin{equation}\label{eqE01}
E_{0,1}(t)\lesssim  E_0(0)+E_0(t)E_1^{1/2}(t)+ E_1^{1/2}(t)e_0(t)+E_0^{5/6}(t)E_1^{1/6}(t)e_0^{1/2}(t).
\end{equation}
Here, we have used the Poincar$\mathrm{\acute{e}}$ inequality to consider the highest order norms only.
$\mathbf{Step\;\; 2}$

Now, let us work for the remaining term in $E_0(t)$. Applying $\nabla^{2s}$ on the second equation of system \eqref{eq1.2} and taking
inner product with $\nabla^{2s} \partial_3 b$,  then multiplying the time weight $(1+t)^{-\sigma}$ we get
\begin{equation}\label{eqlem1.2}
(1+t)^{-\sigma} \|\partial_3 b\|_{\dot H^{2s}}^2 = I_5 + I_6 + I_7 + I_8,
\end{equation}
where
\begin{equation}\nonumber
\begin{split}
I_5 =& (1+t)^{-\sigma} \int_{\mathbb{T}^3} \nabla^{2s}(u\cdot \nabla u)\; \nabla^{2s}\partial_3 b \; dx
- (1+t)^{-\sigma} \int_{\mathbb{T}^3} \nabla^{2s}\Delta u \; \nabla^{2s}\partial_3 b \; dx,\\
I_6 =& -(1+t)^{-\sigma} \sum_{k=0}^{s} \int_{\mathbb{T}^3} \nabla^{k}b_h\cdot \nabla_h \nabla^{2s-k} b \; \nabla^{2s}\partial_3 b
+ \nabla^{k}b_3\cdot \nabla_3 \nabla^{2s-k} b \; \nabla^{2s}\partial_3 b\;dx\\
&- (1+t)^{-\sigma} \sum_{k=s+1}^{2s} \int_{\mathbb{T}^3} \nabla^{k}b_h\cdot \nabla_h \nabla^{2s-k} b \; \nabla^{2s}\partial_3 b
+ \nabla^{k}b_3\cdot \nabla_3 \nabla^{2s-k} b \; \nabla^{2s}\partial_3 b\;dx,\\
I_7 =& \frac{d}{dt}(1+t)^{-\sigma} \int_{\mathbb{T}^3} \nabla^{2s} u \; \nabla^{2s}\partial_3 b\; dx
+ \sigma (1+t)^{-1-\sigma} \int_{\mathbb{T}^3} \nabla^{2s} u \; \nabla^{2s}\partial_3 b\; dx,\\
I_8 =& (1+t)^{-\sigma} \int_{\mathbb{T}^3} \nabla^{2s} \partial_3 u \; \nabla^{2s}\partial_t b\; dx.
\end{split}
\end{equation}

Like the process in Step 1, we shall derive the estimate of each term on the right hand side of \eqref{eqlem1.2}. First, using H\"{o}lder inequality and Sobolev imbedding theorem, we can bound $I_5$ as follows:
\begin{equation}\nonumber
\begin{split}
|I_5| \lesssim & (1+t)^{-\sigma} \|u\|_{H^{2s+1}}\|u\|_{H^{s+2}}\|\partial_3 b\|_{H^{2s}}
+ (1+t)^{-\sigma} \|u\|_{H^{2s+2}} \|\partial_3 b\|_{H^{2s}}\\
\lesssim & (1+t)^{-\sigma} \|u\|_{H^{2s+1}}\| b\|_{H^{2s+1}}\|u\|_{H^{2s-1}}
+ (1+t)^{-\sigma} \|u\|_{H^{2s+2}} \|\partial_3 b\|_{H^{2s}},
\end{split}
\end{equation}
provided that $s\geq 3$. Thus, we have
\begin{equation}\label{eqI5}
\int_{0}^{t} |I_5(\tau)| \; d\tau \lesssim E_0(t)E_1^{1/2}(t) + E_{0,1}^{1/2}(t) \big(\int_{0}^{t} (1+\tau)^{-\sigma} \|\partial_3 b\|_{H^{2s}}^2 d\tau\big )^{1/2}.
\end{equation}

Next, we turn to the estimate of $I_6$. Notice that $I_6$ is the most wild term in our proof, due to the bad  behaviour of $b\cdot\nabla b$. Although we have already divided this term into $b_h \cdot\nabla_h b$ and $b_3\cdot\nabla_3 b$ two terms, the estimate for $b_h\cdot\nabla_h b$ is still nontrivial. Thanks to the Proposition \ref{prop} we have proved in the beginning of this section, we can overcome this problem using the following strategy.

Notice the property \eqref{eq1.3} we easily know that in the $x_3$ direction, the average value of function $b_h(x_h,\cdot)$ over $[-\pi,\pi]$ equals zero. Thus, using the Proposition 2.1, H\"{o}lder inequality and Sobolev imbedding theorem we get
\begin{equation}\nonumber
\begin{split}
|I_6| \lesssim &(1+t)^{-\sigma}(\|b_h\|_{W^{s,\infty}} \|b\|_{H^{2s+1}} \|\partial_3 b\|_{H^{2s}}
+ \|b_3\|_{W^{s,\infty}}\|\partial_3 b\|_{H^{2s}}^2)\\
&+ (1+t)^{-\sigma}(\|b_h\|_{H^{2s}} \|b\|_{W^{s,\infty}} \|\partial_3 b\|_{H^{2s}}
+ \|b_3\|_{H^{2s}}\|\partial_3 b\|_{W^{s-1,\infty}}\|\partial_3 b\|_{H^{2s}})\\
\lesssim &(1+t)^{-\sigma}(\|\partial_3 b_h\|_{H^{s+2}} \|b\|_{H^{2s+1}} \|\partial_3 b\|_{H^{2s}}
+ \|b_3\|_{H^{s+2}}\|\partial_3 b\|_{H^{2s}}^2)\\
&+ (1+t)^{-\sigma}(\|\partial_3 b_h\|_{H^{2s}} \|b\|_{H^{s+2}} \|\partial_3 b\|_{H^{2s}}
+ \|b_3\|_{H^{2s}}\|\partial_3 b\|_{H^{s+1}}\|\partial_3 b\|_{H^{2s}})\\
\lesssim & (1+t)^{-\sigma}(\|\partial_3 b\|_{H^{2s-3}} \|b\|_{H^{2s+1}} \|\partial_3 b\|_{H^{2s}}
+ \|b\|_{H^{2s}}\|\partial_3 b\|_{H^{2s}}^2),
\end{split}
\end{equation}
provided that $s\geq 5$.
Hence,
\begin{equation}\label{eqI6}
\int_{0}^{t} |I_6(\tau)|\; d\tau \lesssim E_0(t)E_1^{1/2}(t) + E_0(t)e_0^{1/2}(t).
\end{equation}

For the next term $I_7$, using H\"{o}lder inequality, it is straightforward to see
\begin{equation}\label{eqI7}
|\int_{0}^{t}I_7 (\tau) \; d\tau| \lesssim E_{0,1}(t).
\end{equation}

For the last term $I_8$, using the first equation of system \eqref{eq1.2} and integrating by parts, we find
\begin{equation}\nonumber
\begin{split}
I_8 =& (1+t)^{-\sigma}\int_{\mathbb{T}^3} \nabla^{2s}\partial_3 u \nabla^{2s}( \partial_3 u+ b\cdot \nabla u - u\cdot \nabla b)\;dx \\
= &-(1+t)^{-\sigma}\int_{\mathbb{T}^3} \nabla^{2s+1}\partial_3 u \nabla^{2s-1}( \partial_3 u+ b\cdot \nabla u - u\cdot \nabla b)\;dx.
\end{split}
\end{equation}
Thus, by H\"{o}lder inequality and Sobolev imbedding theorem, we have
\begin{equation}\nonumber
\begin{split}
|I_8| \lesssim  & (1+t)^{-\sigma} (\|u\|_{H^{2s+2}}^2 + \|u\|_{H^{2s+2}}^2 \|b\|_{H^{2s}}),
\end{split}
\end{equation}
provided that $s\geq 2$.
Hence, we arrive at
\begin{equation}\label{eqI8}
\int_{0}^{t} |I_8(\tau)| \; d\tau \lesssim E_{0,1}(t) + E_0(t)e_0^{1/2}(t).
\end{equation}

Summing up the estimates for $I_5\thicksim I_8$, i.e., \eqref{eqI5}, \eqref{eqI6}, \eqref{eqI7} and \eqref{eqI8}, and integrating \eqref{eqlem1.2} in time, using Young inequality and Poincar$\mathrm{\acute{e}}$ inequality we can easily bound
\begin{equation}\label{eqE02}
\int_{0}^{t}(1+\tau)^{-\sigma}\|\partial_3 b\|_{H^{2s}}^2 \; d\tau \lesssim E_{0,1}+ E_0(t)E_1^{1/2}(t)
+ E_0(t)e_0^{1/2}(t).
\end{equation}

This gives the estimate for the last term in $E_0(t)$. Now, multiplying \eqref{eqE01} by suitable large number and plus \eqref{eqE02}, we then complete the proof of this lemma.

\end{proof}

Next, we work with the lower order energies defined in \eqref{eqframe}, especially we want to derive the decay estimate and get the uniform bound of lower order norms of magnetic field.
\begin{lemma}\label{lem2}
Assume that $s\geq 5$ and the energies are defined as in \eqref{eqframe}, then we have
\begin{equation}\label{G_0}\nonumber
\begin{split}
G_0(t)\lesssim &E_0(t) +  G_0(t) E_1^{1/2}(t)+ E_0^{1/2}(t)G_0^{1/2}(t)[G_1^{1/2}(t)+E_1^{1/2}(t)]\\
& + E_0^{1/2}(t)G_0^{1/4}(t)G_1^{1/4}(t)e_0^{1/2}(t).\\
\end{split}
\end{equation}
\end{lemma}

\begin{proof}
First, applying $\nabla^{2s}\partial_3$ derivative on the system \eqref{eq1.2}. Then, taking inner product with $\nabla^{2s} \partial_3 b$ for the first equation of system \eqref{eq1.2} and taking inner product with $\nabla^{2s} \partial_3 u$ for the second equation of system \eqref{eq1.2}. Summing them up and multiplying the time weight $(1+t)^{1-\sigma}$ we obtain
\begin{equation}\label{eqlem2.1}
\begin{split}
\frac{1}{2}\frac{d}{dt}(1+t)^{1-\sigma}(\|\partial_3 u\|_{\dot H^{2s}}^2+\|\partial_3 b\|_{\dot H^{2s}}^2)+(1+t)^{1-\sigma}\|\partial_3 u\|_{\dot H^{2s+1}}^2=  \sum_{i=1}^{6} J_i,
\end{split}
\end{equation}
where,
\allowdisplaybreaks[2]
\begin{align}
J_1=&\frac{1-\sigma}{2} (1+t)^{-\sigma}(\|\partial_3 u\|_{\dot H^{2s}}^2+\|\partial_3 b\|_{\dot H^{2s}}^2),\nonumber\\
J_2=&-(1+t)^{1-\sigma}\int_{\mathbb{T}^3} u\cdot \nabla \nabla^{2s} \partial_3 u \; \nabla^{2s} \partial_3 u + u\cdot \nabla \nabla^{2s}
\partial_3 b \; \nabla^{2s} \partial_3 b \;dx\nonumber\\
&+(1+t)^{1-\sigma}\int_{\mathbb{T}^3} b\cdot \nabla \nabla^{2s} \partial_3 b \; \nabla^{2s} \partial_3 u + b\cdot \nabla \nabla^{2s} \partial_3 u \; \nabla^{2s} \partial_3 u \;dx\nonumber\\
&+(1+t)^{1-\sigma}\int_{\mathbb{T}^3} \nabla^{2s}\partial_3^2 u \; \nabla^{2s} \partial_3 b + \nabla^{2s}\partial_3^2 b \;
\nabla^{2s} \partial_3 u \;dx,\nonumber\\
J_3=&-(1+t)^{1-\sigma}\sum_{k=1}^{2s}\int_{\mathbb{T}^3} \nabla^{k}u\cdot \nabla \nabla^{2s-k} \partial_3 u\; \nabla^{2s} \partial_3 u \;dx\nonumber\\
&-(1+t)^{1-\sigma}\sum_{k=0}^{2s}\int_{\mathbb{T}^3} \nabla^{k}\partial_3 u\cdot \nabla \nabla^{2s-k}  u\; \nabla^{2s} \partial_3 u \;dx,\nonumber\\
J_4=&-(1+t)^{1-\sigma}\sum_{k=1}^{2s} \int_{\mathbb{T}^3} \nabla^{k} u \cdot \nabla \nabla^{2s-k}\partial_3 b \;\nabla^{2s}\partial_3 b \;dx\nonumber\\
& -(1+t)^{1-\sigma}\sum_{k=0}^{2s} \int_{\mathbb{T}^3} \nabla^{k} \partial_3u \cdot \nabla \nabla^{2s-k} b \;\nabla^{2s}\partial_3 b \;dx,\nonumber\\
J_5=&(1+t)^{1-\sigma}\sum_{k=0}^{2s} \int_{\mathbb{T}^3} \nabla^{k} \partial_3 b \cdot \nabla \nabla^{2s-k} u \;\nabla^{2s}\partial_3 b \;dx\nonumber\\
&+(1+t)^{1-\sigma}\sum_{k=1}^{2s} \int_{\mathbb{T}^3} \nabla^{k} b \cdot \nabla \nabla^{2s-k} \partial_3 u \;\nabla^{2s}\partial_3 b \;dx,\nonumber\\
J_6=&(1+t)^{1-\sigma}\sum_{k=1}^{2s}\int_{\mathbb{T}^3}\nabla^{k}b \cdot \nabla \nabla^{2s-k} \partial_3 b\; \nabla^{2s} \partial_3 u\;dx\nonumber\\
&+(1+t)^{1-\sigma}\sum_{k=0}^{2s}\int_{\mathbb{T}^3}\nabla^{k}\partial_3 b \cdot \nabla \nabla^{2s-k} b\; \nabla^{2s} \partial_3 u\;dx.\nonumber
\end{align}

Like the proof in Lemma \ref{lem1}, we shall now estimate each term on the right hand side of \eqref{eqlem2.1}. First, for the term $J_1$, it is easy to see that

\begin{equation}\label{eqJ1}
\int_{0}^{t} |J_1(\tau)| \; d\tau \lesssim E_0(t).
\end{equation}

Using integration by parts and divergence free condition, it is clear that
\begin{equation}\label{eqJ2}
J_2=0.
\end{equation}

For each term in $J_3$, we divide it into two parts: $k\leq s$ and $k > s$. We treat these two cases respectively and estimate as follows:
\begin{equation}\nonumber
\begin{split}
|J_3| \lesssim &(1+t)^{1-\sigma} (\|u\|_{W^{s,\infty}}\|\partial_3 u\|_{H^{2s}}^2+\|u\|_{H^{2s}}\|\partial_3 u\|_{W^{s,\infty}}\|\partial_3 u\|_{H^{2s}})\\
&+(1+t)^{1-\sigma} (\|\partial_3 u\|_{W^{s,\infty}}\| u\|_{H^{2s+1}}\|\partial_3 u\|_{H^{2s}}+\|\partial_3 u\|_{H^{2s}}^2\| u\|_{W^{s,\infty}})\\
\lesssim & (1+t)^{1-\sigma}( \|u\|_{H^{2s-1}}\|\partial_3 u\|_{H^{2s}}^2+\|u\|_{H^{2s+1}}\|\partial_3 u\|_{H^{2s-1}}\|\partial_3 u\|_{H^{2s}}),
\end{split}
\end{equation}
provided that $s\geq 3$. Thus, we have
\begin{equation}\label{eqJ3}
\begin{split}
\int_{0}^{t} |J_3(\tau)| \; d\tau \lesssim & G_0(t)\cdot \int_{0}^{t} \|u\|_{H^{2s-1}} \; d\tau
+ E_0^{1/2}(t)G_0^{1/2}(t) \int_{0}^{t} (1+\tau)^{1/2}\|\partial_3 u\|_{H^{2s-1}} \; d\tau \\
\lesssim & G_0(t) E_1^{1/2}(t)+ E_0^{1/2}(t)G_0^{1/2}(t)G_1^{1/2}(t).
\end{split}
\end{equation}

The term $J_4$ can be estimated by the same method as in $J_3$ as follows
\begin{equation}\nonumber
\begin{split}
|J_4|\lesssim & (1+t)^{1-\sigma} (\|u\|_{W^{s,\infty}}\|\partial_3 b\|_{H^{2s}}^2
+ \|u\|_{H^{2s}}\|\partial_3 b\|_{W^{s,\infty}}\|\partial_3 b\|_{H^{2s}})\\
&+ (1+t)^{1-\sigma} (\|\partial_3 u\|_{W^{s,\infty}}\| b\|_{H^{2s+1}}\|\partial_3 b\|_{H^{2s}}
+ \|\partial_3 u\|_{H^{2s}}\| b\|_{W^{s,\infty}}\|\partial_3 b\|_{H^{2s}})\\
\lesssim & (1+t)^{1-\sigma} (\|u\|_{H^{2s-1}}\|\partial_3 b\|_{H^{2s}}^2
+ \|u\|_{H^{2s}}\|\partial_3 b\|_{H^{2s-3}}\|\partial_3 b\|_{H^{2s}})\\
&+ (1+t)^{1-\sigma} (\|\partial_3 u\|_{H^{2s-1}}\| b\|_{H^{2s+1}}\|\partial_3 b\|_{H^{2s}}
+ \|\partial_3 u\|_{H^{2s}}\| b\|_{H^{2s}}\|\partial_3 b\|_{H^{2s}}),
\end{split}
\end{equation}
provided that $s\geq 5$. Now,
\begin{equation}\label{eqJ4}
\begin{split}
\int_{0}^{t} |J_4(\tau)| \; d\tau \lesssim &G_0(t) \int_{0}^{t}\|u\|_{H^{2s-1}} \; d\tau + E_0^{1/2}(t) G_0^{1/2}(t)
\int_{0}^{t}(1+\tau)^{1/2} \|\partial_3 b\|_{H^{2s-3}} \; d\tau\\
& + E_0^{1/2}(t) G_0^{1/2}(t)
\int_{0}^{t}(1+\tau)^{1/2} \|\partial_3 u\|_{H^{2s-1}} \; d\tau\\
&+ e_0^{1/2}(t)
\Big(\!  \int_{0}^{t}\!   (1+\tau)^{-\sigma}\|\partial_3 b\|_{H^{2s}}^2 \; d\tau\!  \Big)^{1/2}
\!  \Big(\!  \int_{0}^{t} \!  (1+\tau)^{2-\sigma}\|\partial_3 u\|_{H^{2s}}^2 \; d\tau\!  \Big)^{1/2}\\
\lesssim & G_0(t)E_1^{1/2}(t)+E_0^{1/2}(t)G_0^{1/2}(t)E_1^{1/2}(t)+ E_0^{1/2}(t)G_0^{1/4}(t)G_1^{1/4}(t)e_0^{1/2}(t).
\end{split}
\end{equation}
where, we have used the following inequality
\begin{equation}\label{eqlem2.2}
\begin{split}
\int_{0}^{t} (1+\tau)^{2-\sigma}\|\partial_3 u\|_{H^{2s}}^2 \; d\tau \lesssim &
\int_{0}^{t} (1+\tau)^{\frac{1-\sigma}{2}}\|\partial_3 u\|_{H^{2s+1}}  (1+\tau)^\frac{3-\sigma}{2}\|\partial_3 u\|_{H^{2s-1}}\; d\tau\\
\lesssim & G_0^{1/2}(t) G_1^{1/2}(t).
\end{split}
\end{equation}

Similarly, we can estimate $J_5$ as follows:
\begin{equation}\nonumber
\begin{split}
|J_5| \lesssim & (1+t)^{1-\sigma}( \|\partial_3 b\|_{W^{s,\infty}}\|u\|_{H^{2s+1}}\|\partial_3 b\|_{H^{2s}}
+ \|\partial_3 b\|_{H^{2s}}^2 \|u\|_{W^{s,\infty}})\\
&+(1+t)^{1-\sigma}(\|b\|_{W^{s,\infty}}\|\partial_3 u\|_{H^{2s}} \|\partial_3 b\|_{H^{2s}}
+ \|b\|_{H^{2s}}\|\partial_3 u\|_{W^{s,\infty}}\|\partial_3 b\|_{H^{2s}})\\
\lesssim & (1+t)^{1-\sigma}( \|\partial_3 b\|_{H^{2s-3}}\|u\|_{H^{2s+1}}\|\partial_3 b\|_{H^{2s}}
+ \|\partial_3 b\|_{H^{2s}}^2 \|u\|_{H^{2s-1}})\\
&+(1+t)^{1-\sigma}\|b\|_{H^{2s}}\|\partial_3 u\|_{H^{2s}} \|\partial_3 b\|_{H^{2s}},
\\
\end{split}
\end{equation}
provided that $s\geq 5$.
Hence, using \eqref{eqlem2.2}, we easily get
\begin{equation}\label{eqJ5}
\begin{split}
&\int_{0}^{t} |J_5(\tau)| \; d\tau \\
\lesssim  & E_0^{1/2}(t)G_0^{1/2}(t)\int_{0}^{t} (1+\tau)^{1/2} \|\partial_3 b\|_{H^{2s-3}} \; d\tau
+ G_0(t) \int_{0}^{t} \|u\|_{H^{2s-1}} \; d\tau\\
&\quad+e_0^{1/2}E_0^{1/2}\Big(\int_{0}^{t} (1+\tau)^{2-\sigma}\|\partial_3 u\|_{H^{2s}}^2 \; d\tau\Big)^{1/2}\\
 \lesssim & E_0^{1/2}(t)G_0^{1/2}(t)E_1^{1/2}(t)+G_0(t)E_1^{1/2}(t)+E_0^{1/2}(t)G_0^{1/4}(t)G_1^{1/4}(t)e_0^{1/2}(t).
\end{split}
\end{equation}

In the same manner, we can estimate the last term $J_6$. Indeed,
\begin{equation}\nonumber
\begin{split}
|J_6| \lesssim & (1+t)^{1-\sigma}(\|b\|_{W^{s,\infty}}\|\partial_3 b\|_{H^{2s}}\|\partial_3 u\|_{H^{2s}}
+ \|b\|_{H^{2s}}\|\partial_3 b\|_{W^{s,\infty}}\|\partial_3 u\|_{H^{2s}})\\
 &+ (1+t)^{1-\sigma} (\|\partial_3 b\|_{W^{s,\infty}}\|b\|_{H^{2s+1}}\|\partial_3 u\|_{H^{2s}}+
\|\partial_3 b\|_{H^{2s}} \|b\|_{W^{s,\infty}}\|\partial_3 u\|_{H^{2s}})\\
\lesssim & (1+t)^{1-\sigma}\big(\|b\|_{H^{2s}}\|\partial_3 b\|_{H^{2s}}\|\partial_3 u\|_{H^{2s}}+ \|\partial_3 b\|_{H^{2s-3}}\|b\|_{H^{2s+1}}\|\partial_3 u\|_{H^{2s}}\big),
\end{split}
\end{equation}
provided that $s\geq 5$. Then, it is clear that
\begin{equation}\label{eqJ6}
\begin{split}
\int_{0}^{t} |J_6(\tau)| \; d\tau \lesssim & e_0^{1/2}E_0^{1/2}\Big(\int_{0}^{t} (1+\tau)^{2-\sigma}\|\partial_3 u\|_{H^{2s}}^2 \; d\tau\Big)^{1/2}\\
&+ E_0^{1/2}(t)G_0^{1/2}(t)\int_{0}^{t} (1+\tau)^{1/2} \|\partial_3 b\|_{H^{2s-3}} \; d\tau\\
\lesssim & E_0^{1/2}(t)G_0^{1/4}(t)G_1^{1/4}(t)e_0^{1/2}(t)+  E_0^{1/2}(t)G_0^{1/2}(t)E_1^{1/2}(t).
\end{split}
\end{equation}

Finally, summing up the estimates for $J_1\thicksim J_6$, i.e., \eqref{eqJ1}, \eqref{eqJ2}, \eqref{eqJ3}, \eqref{eqJ4}, \eqref{eqJ5} and \eqref{eqJ6}, and integrating \eqref{eqlem2.1} in time, using Poincar$\mathrm{\acute{e}}$ inequality we can complete the proof of this lemma.
\end{proof}

\begin{lemma}\label{lem3}
Assume that $s\geq 4$ and the energies are defined as in \eqref{eqframe}, then we have
\begin{equation}\nonumber
G_1(t)\lesssim E_0(0) + G_0(t)+ E_0^{1/3}(t)E_1^{2/3}(t)
+  G_1(t)E_1^{1/2}(t)+G_1^{1/2}(t)E_1^{1/2}(t) e_0^{1/2}(t).
\end{equation}
\end{lemma}

\begin{proof}
First, taking $\nabla^{2s-2}\partial_3$ derivative on the system \eqref{eq1.2}. Then, taking inner product with $\nabla^{2s-2} \partial_3 b$ for the first equation of system \eqref{eq1.2} and taking inner product with $\nabla^{2s-2} \partial_3 u$ for the second equation of system \eqref{eq1.2}. Summing them up and multiplying the time weight $(1+t)^{3-\sigma}$ we get
\begin{equation}\label{eqlem3.1}
\begin{split}
\frac{1}{2}\frac{d}{dt}(1+t)^{3-\sigma}(\|\partial_3 u\|_{\dot H^{2s-2}}^2+\|\partial_3 b\|_{\dot H^{2s-2}}^2)+(1+t)^{3-\sigma}\|\partial_3 u\|_{\dot H^{2s-1}}^2=  \sum_{i=1}^{6} N_i,
\end{split}
\end{equation}
where,
\allowdisplaybreaks[2]
\begin{align}
N_1=&\frac{3-\sigma}{2} (1+t)^{2-\sigma}(\|\partial_3 u\|_{\dot H^{2s-2}}^2+\|\partial_3 b\|_{\dot H^{2s-2}}^2),\nonumber\\
N_2=&-(1+t)^{3-\sigma}\int_{\mathbb{T}^3} u\cdot \nabla \nabla^{2s-2} \partial_3 u \; \nabla^{2s-2} \partial_3 u + u\cdot \nabla \nabla^{2s-2}
\partial_3 b \; \nabla^{2s-2} \partial_3 b \;dx\nonumber\\
&+(1+t)^{3-\sigma}\int_{\mathbb{T}^3} b\cdot \nabla \nabla^{2s-2} \partial_3 b \; \nabla^{2s-2} \partial_3 u + b\cdot \nabla \nabla^{2s-2} \partial_3 u \; \nabla^{2s-2} \partial_3 b \;dx\nonumber\\
&+(1+t)^{3-\sigma}\int_{\mathbb{T}^3} \nabla^{2s-2}\partial_3^2 u \; \nabla^{2s-2} \partial_3 b + \nabla^{2s-2}\partial_3^2 b \;
\nabla^{2s-2} \partial_3 u \;dx,\nonumber\\
N_3=&-(1+t)^{3-\sigma}\sum_{k=1}^{2s-2}\int_{\mathbb{T}^3} \nabla^{k}u\cdot \nabla \nabla^{2s-2-k} \partial_3 u\; \nabla^{2s-2} \partial_3 u \;dx\nonumber\\
&-(1+t)^{3-\sigma}\sum_{k=0}^{2s-2}\int_{\mathbb{T}^3} \nabla^{k}\partial_3 u\cdot \nabla \nabla^{2s-2-k}  u\; \nabla^{2s-2} \partial_3 u \;dx,\nonumber\\
N_4=& (1+t)^{3-\sigma} \sum_{k=0}^{2s-2}\int_{\mathbb{T}^3}\nabla^{k} \partial_3 b \cdot \nabla \nabla^{2s-2-k} u\; \nabla^{2s-2}\partial_3 b\; dx\nonumber\\
&-(1+t)^{3-\sigma} \sum_{k=1}^{2s-2} \int_{\mathbb{T}^3}\nabla^{k} u \cdot \nabla \nabla^{2s-2-k} \partial_3 b \; \nabla^{2s-2}\partial_3 b \; dx,\nonumber\\
N_5=& (1+t)^{3-\sigma} \sum_{k=1}^{2s-2}\int_{\mathbb{T}^3}\nabla^{k} b \cdot \nabla \nabla^{2s-2-k}  \partial_3u\; \nabla^{2s-2}\partial_3 b\; dx\nonumber\\
&-(1+t)^{3-\sigma} \sum_{k=0}^{2s-2} \int_{\mathbb{T}^3}\nabla^{k} \partial_3 u \cdot \nabla \nabla^{2s-2-k} b \; \nabla^{2s-2}\partial_3 b \; dx,\nonumber\\
N_6=&(1+t)^{3-\sigma}\sum_{k=1}^{2s-2}\int_{\mathbb{T}^3}\nabla^{k}b \cdot \nabla \nabla^{2s-2-k} \partial_3 b\; \nabla^{2s-2} \partial_3 u\;dx\nonumber\\
&+(1+t)^{3-\sigma}\sum_{k=0}^{2s-2}\int_{\mathbb{T}^3}\nabla^{k}\partial_3 b \cdot \nabla \nabla^{2s-2-k} b\; \nabla^{2s-2} \partial_3 u\;dx.\nonumber
\end{align}

 The first term $N_1$ can be bounded as follows

\begin{equation}\nonumber
\begin{split}
|N_1|\lesssim & (1+t)^{2-\sigma} (\|\partial_3 u\|_{H^{2s}}^2+ \|\partial_3 b\|_{H^{2s-2}}^2) \\
\lesssim & (1+t)^\frac{1-\sigma}{2}\|\partial_3 u\|_{H^{2s+1}}(1+t)^\frac{3-\sigma}{2}\|\partial_3 u\|_{H^{2s-1}}\\
&+ \big [(1+t)^{-\sigma/2}\|\partial_3 b\|_{H^{2s}}\big ]^{2/3}\big[(1+t)^\frac{3-\sigma}{2}\|\partial_3 b\|_{H^{2s-3}}\big]^{4/3},
\end{split}
\end{equation}
and thus
\begin{equation}\label{eqN1}
\int_{0}^{t} |N_1(\tau)| \; d\tau \lesssim G_0^{1/2}(t)G_1^{1/2}(t)+ E_0^{1/3}(t)E_1^{2/3}(t).
\end{equation}

Using integration by parts and divergence free condition, we find
\begin{equation}\label{eqN2}
N_2 = 0.
\end{equation}

For the  term $N_3$, thanks to H\"{o}lder inequality and Sobolev imbedding theorem, we have
\begin{equation}\label{eqN3}
\begin{split}
\int_{0}^{t} |N_3(\tau)| \; d\tau \lesssim &
\int_{0}^{t} (1+\tau)^{3-\sigma}\|\partial_3 u\|_{H^{2s-2}}^2 \|u\|_{H^{2s-1}} \; d\tau \\
\lesssim & G_1(t) \int_{0}^{t} \|u\|_{H^{2s-1}} \; d\tau \\
\lesssim & G_1(t) E_1^{1/2}(t),
\end{split}
\end{equation}
provided that $s\geq 3$.

Then, we turn to the term $N_4$. For each term in $N_4$, we divide it into two parts: $k\leq s-1$ and $k \geq s$. We treat these two cases respectively and estimate as follows:
\begin{equation}\nonumber
\begin{split}
|N_4| \lesssim & (1+t)^{3-\sigma}(\|\partial_3 b\|_{W^{s-1,\infty}}\|u\|_{H^{2s-1}}\|\partial_3 b\|_{H^{2s-2}}
+ \|\partial_3 b\|_{H^{2s-2}}\|u\|_{W^{s-1,\infty}}\|\partial_3 b\|_{H^{2s-2}})\\
&+(1+t)^{3-\sigma}(\|u\|_{W^{s-1,\infty}}\|\partial_3 b\|_{H^{2s-2}}^2 + \|\partial_3 b\|_{W^{s-1,\infty}}\|u\|_{H^{2s-2}}\|\partial_3 b\|_{H^{2s-2}})\\
\lesssim & (1+t)^{3-\sigma}\|u\|_{H^{2s-1}}\|\partial_3 b\|_{H^{2s-2}}^2,
\end{split}
\end{equation}
provided that $s\geq 3$. Indeed,
\begin{equation}\label{eqN4}
\begin{split}
\int_{0}^{t} |N_4(\tau)| \; d\tau \lesssim & G_1(t) \int_{0}^{t} \|u\|_{H^{2s-1}} \; d\tau\\
\lesssim & G_1(t)E_1^{1/2}(t).
\end{split}
\end{equation}

Also for the next term $N_5$, we divide each term into two parts: $k\leq s-1$ and $k \geq s$. Using H\"{o}lder inequality and Sobolev inequality respectively, we can bound
\begin{equation}\nonumber
\begin{split}
|N_5| \lesssim & (1+t)^{3-\sigma} |\sum_{k=1}^{2s-2}\int_{\mathbb{T}^3}\nabla( \nabla^{k} b \cdot \nabla \nabla^{2s-2-k}  \partial_3 u)\; \nabla^{2s-3}\partial_3 b\; dx|\\
&+(1+t)^{3-\sigma}| \sum_{k=0}^{2s-2} \int_{\mathbb{T}^3} \nabla(\nabla^{k} \partial_3 u \cdot \nabla \nabla^{2s-2-k} b) \; \nabla^{2s-3}\partial_3 b \; dx|\\
 \lesssim & (1+t)^{3-\sigma}( \|b\|_{W^{s,\infty}} \|\partial_3 u\|_{H^{2s-1}} \|\partial_3 b\|_{H^{2s-3}}
+ \|b\|_{H^{2s-1}}\|\partial_3 u\|_{W^{s,\infty}}\|\partial_3 b\|_{H^{2s-3}})\\
& +(1+t)^{3-\sigma}(\|\partial_3 u\|_{W^{s,\infty}}\|b\|_{H^{2s}}\|\partial_3 b\|_{H^{2s-3}}
+ \|\partial_3 u\|_{H^{2s-1}}\|b\|_{W^{s,\infty}} \|\partial_3 b\|_{H^{2s-3}})\\
\lesssim & (1+t)^{3-\sigma} \|\partial_3 u\|_{H^{2s-1}}\|\partial_3 b\|_{H^{2s-3}}\|b\|_{H^{2s}},
\end{split}
\end{equation}
provided that $s\geq 3$.
Hence,
\begin{equation}\label{eqN5}
\int_{0}^{t} |N_5(\tau)| \; d\tau \lesssim G_1^{1/2}(t)E_1^{1/2}(t) e_0^{1/2}(t).
\end{equation}

We divide the last term $N_6$ into two parts as follows

\begin{equation}\nonumber
\begin{split}
N_6=&-(1+t)^{3-\sigma}\big\{\sum_{k=2}^{2s-2}\int_{\mathbb{T}^3}\nabla^{k}b \cdot \nabla \nabla^{2s-2-k} \partial_3 b\; \nabla^{2s-2} \partial_3 u\;dx\\
&\qquad+\sum_{k=0}^{2s-3}\int_{\mathbb{T}^3}\nabla^{k}\partial_3 b \cdot \nabla \nabla^{2s-2-k} b\; \nabla^{2s-2} \partial_3 u\;dx\big\}\\
&-\!  (1+t)^{3-\sigma}\big\{\int_{\mathbb{T}^3}\!  \nabla b \cdot \nabla \nabla^{2s-3} \partial_3 b\; \nabla^{2s-2} \partial_3 u\;dx+\int_{\mathbb{T}^3}\!  \nabla^{2s-2}\partial_3 b \cdot \nabla b\; \nabla^{2s-2} \partial_3 u\;dx\big\}\\
\triangleq& N_{6,1} + N_{6,2}.
\end{split}
\end{equation}
For the first part $N_{6,1}$, using H\"{o}lder inequality and Sobolev inequality, we easily get
\begin{equation}\nonumber
|N_{6,1}| \lesssim (1+t)^{3-\sigma}\|b\|_{H^{2s-1}}\|\partial_3 b\|_{H^{2s-3}} \|\partial_3 u\|_{H^{2s-2}},
\end{equation}
provided that $s\geq 4$.
Then for the second part $N_{6,2}$, using integration by parts, we can bound
\begin{equation}\nonumber
\begin{split}
|N_{6,2}| \lesssim &(1+t)^{3-\sigma} \|b\|_{W^{2,\infty}}\|\partial_3 b\|_{H^{2s-3}}\|\partial_3 u\|_{H^{2s-1}}\\
\lesssim & (1+t)^{3-\sigma} \|b\|_{H^{2s}}\|\partial_3 b\|_{H^{2s-3}}\|\partial_3 u\|_{H^{2s-1}},
\end{split}
\end{equation}
provided that $s\geq 2$.
Combining the estimate of $N_{6,1}$ and $N_{6,2}$ together, we finally obtain
\begin{equation}\label{eqN6}
\int_{0}^{t} |N_6(\tau)| \; d\tau \lesssim G_1^{1/2}(t)E_1^{1/2}(t)e_0^{1/2}(t).
\end{equation}

Like the process in above lemmas, according to \eqref{eqN1}, \eqref{eqN2}, \eqref{eqN3}, \eqref{eqN4}, \eqref{eqN5} and \eqref{eqN6}, we complete the proof of this lemma.

\end{proof}

\begin{lemma}\label{lem4}
Assume that $s\geq 4$ and the energies are defined as in \eqref{eqframe}, then we have
\begin{equation}\nonumber
\begin{split}
E_1(t)\lesssim &E_0(t)+ G_0(t) + G_1(t)+E_1^{3/2}(t)+E_1(t)e_0^{1/2}(t)+G_1^{1/2}(t)E_1^{1/2}(t)e_0^{1/2}(t).
\end{split}
\end{equation}
\end{lemma}

\begin{proof}

Like the proof in Lemma \ref{lem1}, we divide the proof into two steps. We first deal with $E_{1,1}(t)$ which defined as follows:
\begin{equation}\label{E11}
E_{1,1}(t):= \sup_{0\leq \tau \leq t} (1+\tau)^{3-\sigma} \|u(\tau)\|_{H^{2s-2}}^2
+ \int_{0}^{t}  (1+\tau)^{3-\sigma}\|u(\tau)\|_{H^{2s-1}}^2 \; d\tau.
\end{equation}
$\mathbf{Step\;\; 1}$

Applying $\nabla^{2s-2}$ on the second equation of system \eqref{eq1.2}. Then, taking inner product with $\nabla^{2s-2} u$ and  multiplying the time weight $(1+t)^{3-\sigma}$ we get
\begin{equation}\label{eqlem4.1}
\frac{1}{2}\frac{d}{dt}(1+t)^{3-\sigma} \|u\|_{\dot H^{2s-2}}^2 +(1+t)^{3-\sigma}\| u\|_{\dot H^{2s-1}}^2= F_1+F_2+F_3+F_4,
\end{equation}
where,
\begin{equation}
\begin{split}
F_1=&\frac{3-\sigma}{2} (1+t)^{2-\sigma}\| u\|_{\dot H^{2s-2}}^2, \\
F_2=&-(1+t)^{3-\sigma}\Big{(}\int_{\mathbb{T}^3} u\cdot \nabla \nabla^{2s-2}  u\nabla^{2s-2}  u \; dx+ \sum_{k=1}^{2s-2}\int_{\mathbb{T}^3}\nabla^{k} u \cdot \nabla \nabla^{2s-2-k}u\nabla^{2s-2}u \; dx \Big{)},\\
F_3=&(1+t)^{3-\sigma}\int_{\mathbb{T}^3}\nabla^{2s-2}\partial_3b  \nabla^{2s-2}  u \; dx,\\
F_4=&(1+t)^{3-\sigma} \int_{\mathbb{T}^3}\nabla^{2s-2}( b \cdot \nabla b)\nabla^{2s-2}u \;dx.\nonumber
\end{split}
\end{equation}

Similarly, we shall estimate each term on right hand side of \eqref{eqlem4.1}. First, for the term $F_1$,
by Gagliardo--Nirenberg interpolation inequality, we have
\begin{equation}\nonumber
\begin{split}
|F_1|\lesssim  &(1+t)^{2-\sigma} \| u\|_{H^{2s}}^2 \\
\lesssim &\big [(1+t)^{-\sigma}\|u\|_{H^{2s+2}}^2\big ]^{1/3}\big[(1+t)^{3-\sigma}\| u\|_{H^{2s-1}}^2\big]^{2/3}.
\end{split}
\end{equation}
Hence,
\begin{equation}\label{eqF1}
\int_{0}^{t} |F_1(\tau)| \; d\tau \lesssim  E_0^{1/3}(t)E_1^{2/3}(t).
\end{equation}

For the term $F_2$, integrating by parts and using the divergence free condition, we directly know the first part of $F_2$ equals $0$. Hence, by H\"{o}lder inequality and Sobolev imbedding theorem, we get
\begin{equation}\label{eqF2}
\begin{split}
  \int_{0}^{t} |F_2(\tau)| \; d\tau \lesssim &\int_0^t (1+\tau)^{3-\sigma}\|u\|_{W^{s-1,\infty}}\|u\|_{H^{2s-2}}^2 \; d\tau \\
   \lesssim &\sup_{0\leq \tau \leq t} (1+\tau)^{3-\sigma}\|u\|_{H^{2s-2}}^2 \int_0^t \|u\|_{H^{2s-1}}\; d\tau \\
   \lesssim & E_1^{3/2}(t).
   \end{split}
\end{equation}
provided that $s\geq 2$.

Next, we turn to the estimate of $F_3$ and $F_4$ which are the wildest terms, due to the bad behaviour of $b$. Thanks to the Proposition \ref{prop}, we can use the same strategy as the estimate of $I_6$ in Lemma \ref{lem1} to solve this problem.

For the term $F_3$, using integration by parts and Proposition \ref{prop}, we get
\begin{equation}\nonumber
\begin{split}
|F_3|\lesssim & (1+t)^{3-\sigma}|\int_{\mathbb{T}^3} \nabla^{2s-3}b_h\nabla^{2s-1}\partial_3u_h - \nabla^{2s-3}\partial_3b_3\nabla^{2s-1}u_3 \; dx| \\
\lesssim &(1+t)^{3-\sigma} \big{(} \|b_h\|_{H^{2s-3}}\|\partial_3u_h\|_{H^{2s-1}}+\|\partial_3b_3\|_{H^{2s-3}}\|u_3\|_{H^{2s-1}} \big{)}\\
\lesssim & (1+t)^{\frac{3-\sigma}{2}} \|\partial_3b\|_{H^{2s-3}} (1+t)^{\frac{3-\sigma}{2}} \|\partial_3u\|_{H^{2s-1}}.
\end{split}
\end{equation}
Hence,
\begin{equation}\label{eqF3}
\int_{0}^{t} |F_3(\tau)| \; d\tau \lesssim  G_1^{1/2}(t)E_1^{1/2}(t).
\end{equation}

Also, for the term $F_4$, using integration by parts and dividing the term into four parts, we have
\begin{equation}\nonumber
  \begin{split}
    F_4 = &- (1+t)^{3-\sigma} \int_{\mathbb{T}^3}\nabla^{2s-3}( b \cdot \nabla b)\nabla^{2s-1}u \;dx \\
   = &-(1+t)^{3-\sigma} \sum_{k=0}^{s-1}\int_{\mathbb{T}^3} \big{(} \nabla^kb_h\cdot\nabla_h\nabla^{2s-3-k}b +\nabla^kb_3\cdot\nabla_3\nabla^{2s-3-k}b \big{)} \nabla^{2s-1}u \;dx\\
    &-(1+t)^{3-\sigma} \sum_{k=s}^{2s-3}\int_{\mathbb{T}^3} \big{(} \nabla^kb_h\cdot\nabla_h\nabla^{2s-3-k}b +\nabla^kb_3\cdot\nabla_3\nabla^{2s-3-k}b \big{)} \nabla^{2s-1}u \;dx.
  \end{split}
\end{equation}
Using H\"{o}lder inequality, Sobolev imbedding theorem and Proposition \ref{prop}, we get
\begin{equation}\nonumber
\begin{split}
 |F_4| \lesssim & (1+t)^{3-\sigma} \big{(}\|b_h\|_{W^{s-1,\infty}}\|b\|_{H^{2s-2}}\|u\|_{H^{2s-1}}+\|b_3\|_{W^{s-1,\infty}}\|\partial_3 b\|_{H^{2s-3}}\|u\|_{H^{2s-1}}\\
  &+\|b_h\|_{H^{2s-3}}\|b\|_{W^{s-2,\infty}}\|u\|_{H^{2s-1}}+\|b_3\|_{H^{2s-3}}\|\partial_3b\|_{W^{s-3,\infty}}\|u\|_{H^{2s-1}} \big{)}\\
  \lesssim & (1+t)^{3-\sigma} \big{(} \|\partial_3b\|_{H^{s+1}}\|b\|_{H^{2s-2}}\|u\|_{H^{2s-1}}+\|b_3\|_{H^{s+1}}\|\partial_3b\|_{H^{2s-3}}\|u\|_{H^{2s-1}}\\
 & +\|\partial_3b\|_{H^{2s-3}}\|b\|_{H^{s}}\|u\|_{H^{2s-1}}+\|b_3\|_{H^{2s-3}}\|\partial_3b\|_{H^{s-1}}\|u\|_{H^{2s-1}})\\
 \lesssim & (1+t)^{3-\sigma} \|b\|_{H^{2s-1}} \|\partial_3 b\|_{H^{2s-3}} \|u\|_{H^{2s-1}},
\end{split}
\end{equation}
provided that $s\geq 4$. Hence,

\begin{equation}\label{eqF4}
\begin{split}
  &\int_{0}^{t} |F_4(\tau)| \; d\tau \\
  \lesssim & \sup_{0\leq \tau \leq t}\|b\|_{H^{2s-1}} \Big{(} \int_0^t (1+\tau)^{3-\sigma}\|\partial_3b\|_{H^{2s-3}}^2\; d\tau\Big{)}^{1/2}\Big{(} \int_0^t (1+\tau)^{3-\sigma}\|u\|_{H^{2s-1}}^2\; d\tau\Big{)}^{1/2} \\
  \lesssim & E_1(t)e_0^{1/2}(t).
\end{split}
\end{equation}

Summing up the estimates for $F_1 \sim F_4$, i.e., \eqref{eqF1}, \eqref{eqF2}, \eqref{eqF3} and \eqref{eqF4}, and integrating \eqref{eqlem4.1} in time, we can get the estimate of $E_{1,1}(t)$ which is defined in \eqref{E11}
\begin{equation}\label{eqE11}
E_{1,1}(t)\lesssim  E_1(0)+E_0(t)^{1/3}E_1^{2/3}(t)+G_1^{1/2}(t)E_1^{1/2}(t)+E_1^{3/2}(t)+E_1(t)e_0^{1/2}(t).
\end{equation}
Here, we have used the Poincar$\mathrm{\acute{e}}$ inequality to consider the highest order norms only.
$\mathbf{Step\;\; 2}$

Now, let us work for the remaining term in $E_1(t)$. Applying $\nabla^{2s-3}$ derivative on the second equation of system \eqref{eq1.2} and taking
inner product with $\nabla^{2s-3} \partial_3 b$, multiplying the time-weight $(1+t)^{3-\sigma}$ we get

\begin{equation}\label{eqlem4.2}
(1+t)^{3-\sigma} \|\partial_3 b\|_{\dot H^{2s-3}}^2 = F_5 + F_6 + F_7 + F_8,
\end{equation}
where
\begin{equation}\nonumber
\begin{split}
F_5 =& (1+t)^{3-\sigma} \int_{\mathbb{T}^3} \nabla^{2s-3}(u\cdot \nabla u)\nabla^{2s-3}\partial_3 b \; dx
- (1+t)^{3-\sigma} \int_{\mathbb{T}^3} \nabla^{2s-3}\Delta u \nabla^{2s-3}\partial_3 b \; dx,\\
F_6 =& \!  - \!  (1+t)^{3-\sigma} \sum_{k=0}^{s-1} \int_{\mathbb{T}^3} \!  \!  \nabla^{k}b_h\cdot \nabla_h \nabla^{2s-3-k} b \nabla^{2s-3}\partial_3 b
+ \nabla^{k}b_3\cdot \nabla_3 \nabla^{2s-3-k} b \nabla^{2s-3}\partial_3 b\;dx\\
&\!  -\!   (1+t)^{3-\sigma} \sum_{k=s}^{2s-3} \int_{\mathbb{T}^3}\!  \!  \!  \!   \nabla^{k}b_h\cdot \nabla_h \nabla^{2s-3-k} b \nabla^{2s-3}\partial_3 b
+ \!  \nabla^{k}b_3\cdot \nabla_3 \nabla^{2s-3-k} b  \nabla^{2s-3}\partial_3 b\;dx,\\
F_7 =& \frac{d}{dt}(1+t)^{3-\sigma} \int_{\mathbb{T}^3} \nabla^{2s-3} u  \nabla^{2s-3}\partial_3 b\; dx
-(3-\sigma )(1+t)^{2-\sigma} \int_{\mathbb{T}^3} \nabla^{2s-3} u \nabla^{2s-3}\partial_3 b\; dx,\\
F_8 =& (1+t)^{3-\sigma} \int_{\mathbb{T}^3} \nabla^{2s-3} \partial_3 u  \nabla^{2s-3}\partial_t b\; dx.
\end{split}
\end{equation}

Similar to the process in Step 1, we shall drive the estimate of each term on the right hand side of \eqref{eqlem4.2}.
First, using H\"{o}lder inequality and Sobolev imbedding theorem, we get
\begin{equation}\nonumber
|F_5| \lesssim (1+t)^{3-\sigma} \|u\|_{H^{2s-2}}\|u\|_{H^{s+2}}\|\partial_3 b\|_{H^{2s-3}}
+ (1+t)^{3-\sigma} \|u\|_{H^{2s-1}} \|\partial_3 b\|_{H^{2s-3}}.
\end{equation}
Hence, for $s\geq 3$,
\begin{equation}\label{eqF5}
\int_{0}^{t} |F_5(\tau)| \; d\tau \lesssim E_1^{3/2}(t) + E_{1,1}^{1/2}(t) \big[\int_{0}^{t}(1+\tau)^{3-\sigma}\|\partial_3 b\|_{H^{2s-3}}^2 \;d\tau \big]^{1/2}.
\end{equation}

Next, for the most wild term $F_6$, similar to the estimate of $I_6$ in Lemma \ref{lem1},  we use property \eqref{eq1.3} and Proposition \ref{prop}  to obtain
\begin{equation}\nonumber
\begin{split}
|F_6| \lesssim &(1+t)^{3-\sigma}(\|b_h\|_{W^{{s-1},\infty}} \|b\|_{H^{2s-2}} \|\partial_3 b\|_{H^{2s-3}}
+ \|b_3\|_{W^{{s-1},\infty}}\|\partial_3 b\|_{H^{2s-3}}^2)\\
&+ (1+t)^{3-\sigma}(\|b_h\|_{H^{2s-3}} \|b\|_{W^{{s-2},\infty}} \|\partial_3 b\|_{H^{2s-3}}
+ \|b_3\|_{H^{2s-3}}\|\partial_3 b\|_{W^{s-3,\infty}}\|\partial_3 b\|_{H^{2s-3}})\\
\lesssim &(1+t)^{3-\sigma}(\|\partial_3 b_h\|_{H^{s+1}} \|b\|_{H^{2s-2}} \|\partial_3 b\|_{H^{2s-3}}
+ \|b_3\|_{H^{s+1}}\|\partial_3 b\|_{H^{2s-3}}^2)\\
&+ (1+t)^{3-\sigma}(\|\partial_3 b_h\|_{H^{2s-3}} \|b\|_{H^{s}} \|\partial_3 b\|_{H^{2s-3}}
+ \|b_3\|_{H^{2s-3}}\|\partial_3 b\|_{H^{s-1}}\|\partial_3b\|_{H^{2s-3}})\\
\lesssim & (1+t)^{3-\sigma}\|\partial_3 b\|_{H^{2s-3}}^2 \|b\|_{H^{2s-2}},
\end{split}
\end{equation}
provided that $s\geq 4$.
Hence,
\begin{equation}\label{eqF6}
\int_{0}^{t} |F_6(\tau)|\; d\tau \lesssim E_1(t)e_0^{1/2}(t).
\end{equation}

And, for the term $F_7$,  by H\"{o}lder inequality, we can get
\begin{equation}\label{eqF7}
\begin{split}
\int_{0}^{t} |F_7 (\tau)| \; d\tau \lesssim & G_1^{1/2}(t)E_1^{1/2}(t)+\int_0^t (1+\tau)^\frac{1-\sigma}{2}\|u\|_{H^{2s-3}}(1+\tau)^\frac{3-\sigma}{2}\|\partial_3b\|_{H^{2s-3}} \; d\tau \\
\lesssim & G_1^{1/2}(t)E_1^{1/2}(t)+E_1(t)^{1/2} \Big{(} \int_0^t (1+\tau)^{1-\sigma}\|u\|_{H^{2s+1}}^2 \; d\tau \Big{)}^{1/2} \\
\lesssim & G_1^{1/2}(t)E_1^{1/2}(t)+E_0^{1/3}(t)E_1^{2/3}(t).
\end{split}
\end{equation}

For the last term $F_8$, using the first equation of system \eqref{eq1.2}, we can write
\begin{equation}\nonumber
F_8 =(1+t)^{3-\sigma}\int_{\mathbb{T}^3} \nabla^{2s-3}\partial_3 u \nabla^{2s-3}( \partial_3 u+ b\cdot \nabla u - u\cdot \nabla b)\;dx.
\end{equation}
By H\"{o}lder inequality and Sobolev imbedding theorem, we have
\begin{equation}\nonumber
|F_8| \lesssim   (1+t)^{3-\sigma} (\|\partial_3u\|_{H^{2s-3}}^2 + \|\partial_3u\|_{H^{2s-3}} \|b\|_{H^{2s-2}}\|u\|_{H^{2s-2}}),
\end{equation}
provided that $s\geq 3$.
Hence,
\begin{equation}\label{eqF8}
\int_{0}^{t} |F_8(\tau)| \; d\tau \lesssim E_{1,1}(t) + e_0(t)^{1/2}E_1(t)^{1/2}G_1(t)^{1/2}.
\end{equation}

Summing up the estimates for $F_5 \sim F_8$, i.e., \eqref{eqF5}, \eqref{eqF6}, \eqref{eqF7} and \eqref{eqF8}, and integrating \eqref{eqlem4.2} in time, using Young inequality, we easily get
\begin{equation}\label{eqE12}
\begin{split}
&\int_{0}^{t}(1+\tau)^{3-\sigma}\|\partial_3 b\|_{H^{2s-3}}^2 \; d\tau\\
 \lesssim &E_{1,1}(t)+ E_1^{3/2}(t)+E_1(t)e_0^{1/2}(t)+G_1^{1/2}(t)E_1^{1/2}(t)\\
 &+E_0^{1/3}(t)E_1^{2/3}(t)+G_1^{1/2}(t)E_1^{1/2}(t)e_0^{1/2}(t).
 \end{split}
\end{equation}

This gives the estimate for the last term in $E_1(t)$.
Now, multiplying \eqref{eqE11} by suitable large number and plus \eqref{eqE12}, using Young inequality, we complete the proof of this lemma.
\end{proof}

\begin{lemma}\label{lem5}
Assume that $s\geq 3$ and the energies are defined as in \eqref{eqframe}, then we have
\begin{equation}\nonumber
\begin{split}
e_0(t) \lesssim & E_0(0) + G_0(t) + G_1(t) + E_0^{1/6}(t)E_1^{1/3}(t)e_0(t)\\
&+ E_0^{1/2}(t)E_1^{1/2}(t)e_0^{1/2}(t)+ G_0^{1/4}(t)G_1^{1/4}(t)e_0(t).
\end{split}
\end{equation}
\end{lemma}

\begin{proof}
Taking $\nabla^{2s}$ derivative on the first equation of system \eqref{eq1.2}. Then, taking inner product with $\nabla^{2s} b$ , we get
\begin{equation}\label{eqe0}
\frac{1}{2}\frac{d}{dt}\|b\|_{\dot H^{2s}}^2 = M_1 + M_2 + M_3,
\end{equation}
where,
\begin{equation}\nonumber
\begin{split}
M_1 =&  \sum_{k=1}^{s} \int_{\mathbb{T}^3} ( \nabla^{k} b \cdot \nabla \nabla^{2s-k} u-\nabla^{k} u \cdot \nabla \nabla ^{2s-k} b )\; \nabla^{2s} b \; dx\\
&+\sum_{k=s+1}^{2s} \int_{\mathbb{T}^3} ( \nabla^{k} b \cdot \nabla \nabla^{2s-k} u-\nabla^{k} u \cdot \nabla \nabla ^{2s-k} b )\; \nabla^{2s} b \; dx,\\
M_2 =& \int_{\mathbb{T}^3} (b_h\cdot \nabla_h \nabla^{2s} u + b_3\cdot \nabla_3 \nabla^{2s} u )\nabla^{2s} b\; dx,\\
M_3 =& \int_{\mathbb{T}^3} \nabla^{2s}\partial_3 u\; \nabla^{2s} b\; dx.
\end{split}
\end{equation}
Now we will estimate each term on the right hand side of \eqref{eqe0} line by line.

First, by H\"{o}lder inequality and Sobolev imbedding theorem, we easily get
\begin{equation}\nonumber
\begin{split}
|M_1| \lesssim & \|b\|_{W^{s,\infty}}\|u\|_{H^{2s}}\|b\|_{H^{2s}}+ \|u\|_{W^{s,\infty}}\|b\|_{H^{2s}}^2\\
&+\|b\|_{H^{2s}}\|u\|_{W^{s,\infty}}\|b\|_{H^{2s}}+ \|u\|_{H^{2s}}\|b\|_{W^{s,\infty}}\|b\|_{H^{2s}}\\
\lesssim & \|u\|_{H^{2s}}\|b\|_{H^{2s}}^2,
\end{split}
\end{equation}
provided that $s\geq 2$. Hence, using Gagliardo--Nirenberg interpolation inequality and H\"{o}lder inequality, we can bound
\begin{equation}\label{eqM1}
\int_{0}^{t} |M_1(\tau)|\; d\tau \lesssim e_0(t) \int_{0}^{t}  \|u\|_{H^{2s+2}}^{1/3}\|u\|_{H^{2s-1}}^{2/3}\; d\tau
\lesssim E_0^{1/6}(t)E_1^{1/3}(t)e_0(t).
\end{equation}

For the next term $M_2$, using the same method as above, we directly obtain
\begin{equation}\nonumber
\begin{split}
|M_2| \lesssim &\|b_h\|_{L^{\infty}} \|u\|_{H^{2s+1}} \|b\|_{H^{2s}}+\|b_3\|_{L^\infty}\|\partial_3 u\|_{H^{2s}}\|b\|_{H^{2s}}\\
\lesssim & \| b_h\|_{H^{2s-3}} \|u\|_{H^{2s+1}} \|b\|_{H^{2s}}+ \|\partial_3 u\|_{H^{2s}}\|b\|_{H^{2s}}^2,
\end{split}
\end{equation}
provided that $s\geq 3$. According to the Proposition \ref{prop}, we can use the same strategy as the estimate of $I_6$ in Lemma \ref{lem1}, and obtain
\begin{equation}\nonumber
|M_2| \lesssim \|\partial_3 b_h\|_{H^{2s-3}} \|u\|_{H^{2s+1}} \|b\|_{H^{2s}}+ \|\partial_3 u\|_{H^{2s}}\|b\|_{H^{2s}}^2.
\end{equation}
Using \eqref{eqlem2.2} and H\"{o}lder inequality, we get
\begin{equation}\label{eqM2}
\begin{split}
\int_{0}^{t} |M_2(\tau)| \; d\tau \lesssim &
E_0^{1/2}(t)e_0^{1/2}(t)\int_0^{t} (1+\tau)^{\sigma/2}\|\partial_3 b\|_{H^{2s-3}} \; d\tau +
e_0(t) \int_0^{t} \|\partial_3 u\|_{H^{2s}} \; d\tau \\
\lesssim &
E_0^{1/2}(t)E_1^{1/2}(t)e_0^{1/2}(t)+G_0^{1/4}(t)G_1^{1/4}(t)e_0(t).
\end{split}
\end{equation}

For the last term $M_3$, also we have
\begin{equation}\label{eqM3}
\int_{0}^{t} |M_3(\tau)|\; d\tau \lesssim e_0^{1/2}(t)\cdot \int_{0}^{t} \|\partial_3 u\|_{H^{2s}}\; d\tau
\lesssim G_0^{1/4}(t)G_1^{1/4}(t)e_0^{1/2}(t).
\end{equation}

Combining \eqref{eqM1}, \eqref{eqM2} and \eqref{eqM3} together, we now complete the proof of this lemma by using Young inequality.
\end{proof}

\subsection{Proof of the Theorem \ref{thm1}}
Now, let us combine the above $a \ priori$ estimates of all the energies defined in \eqref{eqframe} together, and finally give the proof of Theorem \ref{thm1}. First, we define the total energy as follows:
$$E_{\text{total}}(t)=E_0(t) + G_0(t) + G_1(t) + E_1(t) + e_0(t). $$
Then, multiplying each inequality in the above five lemmas by  different suitable number, and summing them up, we can obtain the following inequality
\begin{equation}\label{eqEtotal1}
 E_{\text{total}}(t)\leq C_1 E_0(0) + C_1 E_{\text{total}}^{3/2}(t),
 \end{equation}
for some positive constant $C_1$.

According to the setting of initial data in Theorem \ref{thm1}, there exists a positive constant $C_2$ such that
$ E_{\text{total}} (0) + C_1 E_0(0)\leq C_2 \varepsilon $. Due to the local existence result which can be achieved through basic energy method, there exists a positive time $T$ such that
\begin{equation}\label{eqEtotal2}
 E_{\text{total}} (t) \leq 2 C_2\varepsilon , \quad  \forall \; t \in [0, T].
\end{equation}
Let $T^{*}$ be the largest possible time of $T$ for what \eqref{eqEtotal2} holds, then we only need to show $T^{*} = \infty$ while completing the proof of Theorem \ref{thm1}. Notice the estimate \eqref{eqEtotal1}, we can use
 a standard continuation argument to show that $T^{*} = \infty$ provided that $\epsilon$ is small enough.  We omit the details here. Hence, we finish the proof of Theorem \ref{thm1}.

\section*{Acknowledgement}

The authors are grateful to the careful reviewers for constructive comments.
The first author is partially supported by NSF under grant DMS-1516415.
The second author is supported by Key Laboratory of Mathematics for Nonlinear Sciences (Fudan University), Ministry of Education of China, Shanghai, Key Laboratory for Contemporary Applied Mathematics, School of Mathematical Sciences, Fudan University, NSFC under grant No.11421061, 973 Program (grant No.2013CB834100) and 111 project.

\subsection*{Conflict of interest}
The authors declared that they have no conflict of interest to this work.

\end{document}